\documentclass[reqno]{amsart}
\usepackage{amsmath} 

\usepackage{amsthm}
\usepackage{amssymb}
\usepackage{graphics}
\usepackage{latexsym}

\numberwithin{equation}{section}
\newtheorem{thm}{Theorem}[section]

\newcommand{\LR}[1]{{\langle {#1} \rangle }}

\newcommand{\EQ}[1]{\begin{equation} \begin{split} #1
 \end{split} \end{equation}}
\newcommand{\EQS}[1]{\begin{align} #1 \end{align}}
\newcommand{\EQQS}[1]{\begin{align*} #1 \end{align*}}

\title[LWP for HWS equation]{A remark on the Half wave Schr\"odinger equation in the energy space} 

\author[I. Kato]{Isao Kato} 
\address[Isao Kato]{Department of Mathematics, Graduate School of Science, Kyoto University,
Kitashirakawa Oiwake-cho, Sakyo-ku, Kyoto, 606-8502, Japan} 
\email[Isao Kato]{kato.isao.23n@st.kyoto-u.ac.jp}


\begin{document}

\begin{abstract}
 We investigate the Cauchy problem for the half wave Schr\"odinger equation in the energy space. 
We derive the local well-posedness in the energy space for the odd power type nonlinearities under certain additional assumption for the initial data, namely $\hat{u}_0 \in L^1_{\xi, \eta}(\mathbb{R}^2)$.  
 
\end{abstract}

\maketitle
\setcounter{page}{001}

\section{Introduction} 
We study the Cauchy problem for the following equation: 
\EQ{ \label{NLS} 
 i\partial_t u + \partial_x^2 u - |D_y| u &= \mu |u|^{p-1}u, \qquad (t, x, y) \in [-T, T] \times \mathbb{R}^2, \\ 
 u(0, x, y) &= u_0(x, y) \in H^{s_1, s_2}(\mathbb{R}^2), 
}
where $|D_y| = (-\partial_y^2)^{\frac{1}{2}}, \mu = \pm 1, p > 1, T > 0$ and $s_1, s_2 \in \mathbb{R}$. 
Also, we define the anisotropic Sobolev spaces $H^{s_1, s_2}(\mathbb{R}^2)$ as 
\EQQS{ 
 H^{s_1, s_2}(\mathbb{R}^2) &= \{ f \in \mathcal{S}'(\mathbb{R}^2)\, ;\, \|f\|_{H^{s_1, s_2}} < \infty \}, \\ 
 \|f\|_{H^{s_1, s_2}} &:= \Bigl( \int_{\mathbb{R}^2} \LR{\xi}^{2s_1} \LR{\eta}^{2s_2} |\hat{f}(\xi, \eta)|^2 \, \mathrm{d}\xi \, \mathrm{d}\eta \Bigr)^{\frac{1}{2}},\qquad \LR{\cdot} := \sqrt{1 + |\cdot |^2}. 
}
\eqref{NLS} with $\mu = 1, p = 3$ is firstly considered by Xu \cite{Xu} in the analysis of large time behavior of the solution for smooth data.  
After \cite{Xu}, Bahri, Ibrahim and Kikuchi \cite{BIK} obtained the local well-posedness for rough data, namely $s_1 = 0, (1 >) s_2 > \frac{1}{2}$ and $1 < p \leqslant 5$ by the fixed point argument. 
\eqref{NLS} has the following conservation laws (the mass $M(u)$ and the energy $E(u)$): 
\EQQS{
 M(u) &= \int_{\mathbb{R}^2} |u|^2 \, \mathrm{d}x\, \mathrm{d}y, \\ 
 E(u) &= \frac{1}{2} \int_{\mathbb{R}^2} \left( |\partial_x u|^2 + |D_y|u \cdot \bar{u} \right) \mathrm{d}x \, \mathrm{d}y - \frac{\mu}{p+1} \int_{\mathbb{R}^2} |u|^{p + 1}\, \mathrm{d}x\, \mathrm{d}y.  
}
Hence the energy space $E$ for \eqref{NLS} lies in $H^{1, 0}(\mathbb{R}^2) \cap H^{0, \frac{1}{2}}(\mathbb{R}^2)$ 
equipped with the norm 
\begin{align*}
 \|u\|_E := \left(\|\partial_x u\|_{L^2}^2 + \||D_y|^{\frac{1}{2}} u\|_{L^2}^2 + \|u\|_{L^2}^2  \right)^{\frac{1}{2}}.  
\end{align*} 
The well-posedness in $E$ is unknown. 
In Proposition 5.4 \cite{BIK}, the Strichartz estimates require $s_2 > \frac{1}{2}$, hence we cannot apply it for the problem. 
Moreover, we cannot apply the Yudovich argument directly to prove uniqueness for \eqref{NLS}. 
If we apply the argument, we need to show $\|u\|_{L^q_{x, y}} < \infty$ for $q$ large enough. 
However, the Gagliardo-Nirenberg inequality shows that it only holds for $(2 <) q < 6$.      
These are the main obstacles to obtain well-posedness in $E$. 

In this paper, we verify the following local well-posedness result in $E = H^{1, 0}(\mathbb{R}^2) \cap H^{0, \frac{1}{2}}(\mathbb{R}^2)$ under additional assumption on the initial data.   
\begin{thm} \label{mth} 
 Let $p = 2k + 1, k \in \mathbb{N}, \mu = \pm 1$. 
 Suppose that $u_0 \in E$ and $\hat{u}_0 \in L^1_{\xi, \eta}(\mathbb{R}^2)$. 
 Then, \eqref{NLS} is locally well-posed. 
 More precisely, there exists $T = T(\|u_0\|_E, \|\hat{u}_0\|_{L^1_{\xi, \eta}}) > 0$ such that 
 \eqref{NLS} has a unique solution $u$ satisfying  
 \EQQS{
  u \in C([-T, T] \, ;\, E) \cap L^{\infty}_{T, x, y} \quad \text{and}\quad \hat{u} \in L^{\infty}_T L^1_{\xi, \eta}.  
 }  
\end{thm} 

$\hat{u}_0 \in L^1_{\xi, \eta}(\mathbb{R}^2)$ in Theorem \ref{mth} seems to be somewhat extra assumption.    
However in order to control $\|u\|_{L^{\infty}_{T, x, y}}$, we need this assumption.  
Also, we suppose the power $p$ is odd to estimate the Duhamel term, see section 2 for details. 
We remark that in Theorem 1.1 \cite{Ka}, the norm inflation (ill-posedness) holds in $E$ for $p > 5$, however in our Theorem \ref{mth}, if we additionally suppose $\hat{u}_0 \in L^1_{\xi, \eta}(\mathbb{R}^2)$, then (local) well-posedness in $E$ holds even if $p > 5$ provided that $p$ is odd.

\section{Proof of Theorem \ref{mth}} 
In this section, we prove Theorem \ref{mth}. 
For a Banach space $X$ and $r > 0$, we define $B_r(X) := \{ f \in X \, ;\, \|f\|_X \leqslant r \}$. 
Throughout the paper, $\hat{\cdot}\, $ denotes the Fourier transform with respect to spatial variables $x$ and  $y$.     
\begin{proof}[Proof of Theorem \ref{mth}]
We prove the local well-posedness for \eqref{NLS} with initial data $u_0 \in E$ and $\hat{u}_0 \in L^1_{\xi, \eta}(\mathbb{R}^2)$ by the fixed point argument. 
By the Duhamel formula, 
\EQS{ \label{Duhamel}
 \Phi (u) = S(t)u_0 - i\mu \int_0^t S(t - \tau) (|u|^{p-1}u)(\tau)\, \mathrm{d}\tau,  
} 
where $S(t) := \text{exp}\{ it(\partial_x^2 - |D_y|) \}$ be the $L^2_{x, y}$ unitary operator for \eqref{NLS}. 
From \eqref{Duhamel},    
\EQQS{ 
 \widehat{\Phi (u)} = e^{-it(\xi^2 + |\eta|)}\hat{u}_0 - i\mu \int_0^t e^{-(t - \tau)(\xi^2 + |\eta|)} \widehat{|u|^{p-1}u}(\tau) \, \mathrm{d}\tau. 
} 
Let $Y := \{ u \in C([-T, T];\, E) \cap L^{\infty}_{T, x, y} \, ;\, \hat{u} \in L^{\infty}_T L^1_{\xi, \eta}\ \text{and}\ \|u\|_Y < \infty \}$ endowed with the norm 
\EQQS{
 \|u\|_Y := \|u\|_{L^{\infty}_T E} + \|u\|_{L^{\infty}_{T, x, y}} + \|\hat{u}\|_{L^{\infty}_T L^1_{\xi, \eta}}. 
}
Let us verify $\Phi$ is a contraction map in $Y$. 
Firstly, we show $\Phi$ is a map in $Y$. 
Suppose that $u_0 \in B_{\delta}(E), \hat{u}_0 \in B_{\delta}(L^1_{\xi, \eta})$ and $u \in B_r(Y)$. 
Then by $\widehat{S(t) u_0} = e^{-it(\xi^2 + |\eta|)}\hat{u}_0$, it is clear that 
\EQS{ \label{linear}
 \|S(t) u_0\|_Y 
 &= \|S(t) u_0\|_{L^{\infty}_T E} + \|S(t) u_0\|_{L^{\infty}_{T, x, y}} + \|\widehat{S(t) u_0}\|_{L^{\infty}_T L^1_{\xi, \eta}} \notag \\  
 &\leqslant \|u_0\|_E + 2\|\hat{u}_0\|_{L^1_{\xi, \eta}} \leqslant 3 \delta.  
} 
The Duhamel term is estimated as follows. 
Since $S$ is the unitary operator in $L^2_{x, y}(\mathbb{R}^2)$ and $p = 2k + 1, k \in \mathbb{N}$, 
we obtain 
\EQS{ \label{Duhamel_E}
 \Bigl\| \int_0^t S(t - \tau) (|u|^{p - 1}u)(\tau)\, \mathrm{d} \tau \Bigr\|_{L^{\infty}_T E} 
 &\leqslant \int_0^T \|S(t - \tau) (|u|^{p - 1}u)(\tau)\|_{L^{\infty}_T E}\, \mathrm{d}\tau \notag \\ 
 &\leqslant CT \||u|^{p - 1}u \|_{L^{\infty}_T E} \notag \\ 
 &\leqslant CT \|u\|_{L^{\infty}_{T, x, y}}^{p - 1} \|u\|_{L^{\infty}_T E} 
  \leqslant CT r^p. 
}
By $p = 2k + 1, k \in \mathbb{N}$ and the Young inequality, we have 
\EQS{ \label{Duhamel_L^infty}
 \Bigl\| \int_0^t S(t - \tau) (|u|^{p - 1}u)(\tau)\, \mathrm{d} \tau \Bigr\|_{L^{\infty}_{T, x, y}} 
 &\leqslant CT \|\widehat{|u|^{p - 1}u}\|_{L^{\infty}_T L^1_{\xi, \eta}} 
  \leqslant CT \|\hat{u}\|_{L^{\infty}_T L^1_{\xi, \eta}}^p 
  \leqslant CTr^p.   
}
Again by $p = 2k + 1, k \in \mathbb{N}$ and the Young inequality lead 
\EQS{ \label{Duhamel_L^1} 
 \Bigl\| \int_0^t e^{-i(t - \tau)(\xi^2 + |\eta|)} \widehat{|u|^{p - 1}u}(\tau) \, \mathrm{d} \tau \Bigr\|_{L^{\infty}_T L^1_{\xi, \eta}} 
 \leqslant CT \| \widehat{|u|^{p - 1}u}\|_{L^{\infty}_T L^1_{\xi, \eta}} 
 \leqslant CT r^p.  
}
From \eqref{linear}--\eqref{Duhamel_L^1}, if we take 
$\delta, T > 0$ such that $3 \delta \leqslant \frac{1}{2}r$ and $3CTr^p \leqslant \frac{1}{2}r$, 
then $\Phi$ is a map in $Y$. 

Next, we show the contraction of $\Phi$. 
Set $u, v \in B_r(Y)$. 
Then from $p = 2k + 1, k \in \mathbb{N}$, we have 
\EQQS{
 &\Bigl\| \int_0^t S(t - \tau) (|u|^{p - 1}u - |v|^{p - 1}v)(\tau)\, \mathrm{d}\tau \Bigr\|_{L^{\infty}_T E} 
  \leqslant CT \||u|^{2k}u - |v|^{2k}v\|_{L^{\infty}_T E}.  
} 
By induction, we easily check 
\EQQS{
 \| |u|^{2k}u - |v|^{2k}v \|_{L^{\infty}_T E} \leqslant (2k + 1)^2 r^{2k} \|u - v\|_{L^{\infty}_T E \, \cap \,  L^{\infty}_{T, x, y}}. 
}
Hence we obtain 
\EQS{ \label{Duhamel_differnce_E}
 \Bigl\| \int_0^t S(t - \tau)(|u|^{p - 1}u - |v|^{p - 1}v)(\tau)\, \mathrm{d}\tau \Bigr\|_{L^{\infty}_T E} 
 &\leqslant CT(2k + 1)^2 r^{2k} \|u - v\|_{L^{\infty}_T E \, \cap \, L^{\infty}_{T, x, y}} \notag \\ 
 &\leqslant  CT(2k + 1)^2 r^{2k} \|u - v\|_Y.                         
}
From $p = 2k + 1, k \in \mathbb{N}$, we have 
\EQS{
 \Bigl\| \int_0^t S(t - \tau)(|u|^{p - 1}u - |v|^{p - 1}v)(\tau)\, \mathrm{d}\tau \Bigr\|_{L^{\infty}_{T, x, y}} 
  \leqslant CT \| \mathcal{F}_{x, y}[|u|^{2k}u - |v|^{2k}v] \|_{L^{\infty}_T L^1_{\xi, \eta}}.  
     \label{hat_difference}                 
}
By induction and the Young inequality, we see 
\EQS{ \label{hat_difference_est}
 \| \mathcal{F}_{x, y}[ |u|^{2k}u - |v|^{2k}v ]\|_{L^{\infty}_T L^1_{\xi, \eta}} 
 \leqslant (2k + 1)r^{2k} \|\hat{u} - \hat{v}\|_{L^{\infty}_T L^1_{\xi, \eta}}. 
}
From \eqref{hat_difference} and \eqref{hat_difference_est}, we have 
\EQS{ \label{Duhamel_differnce_L^infty} 
 \Bigl\| \int_0^t S(t - \tau)(|u|^{p - 1}u - |v|^{p - 1}v)(\tau)\, \mathrm{d}\tau \Bigr\|_{L^{\infty}_{T, x, y}} 
  \leqslant CT(2k + 1)^2 r^{2k}\|u - v\|_Y.  
}
From \eqref{hat_difference}--\eqref{Duhamel_differnce_L^infty}, we also obtain 
\EQS{ \label{Duhamel_difference_L^1} 
 \Bigl\| \int_0^t e^{-i(t - \tau)(\xi^2 + |\eta|)} \mathcal{F}_{x, y}[|u|^{p - 1}u - |v|^{p - 1}v](\tau) \, \mathrm{d} \tau \Bigr\|_{L^{\infty}_T L^1_{\xi, \eta}} 
 &\leqslant CT \|\mathcal{F}_{x, y}[|u|^{p - 1}u - |v|^{p - 1}v]\|_{L^{\infty}_T L^1_{\xi, \eta}} \notag \\ 
 &\leqslant CT(2k + 1)^2 r^{2k}\|u - v\|_Y. 
}
Collecting \eqref{Duhamel_differnce_E}, \eqref{Duhamel_differnce_L^infty}, \eqref{Duhamel_difference_L^1} 
and taking $T > 0$ such that $3CT(2k + 1)^2 r^{2k} = \frac{1}{2}$, then 
\EQQS{
 \Bigl\| \int_0^t S(t - \tau)(|u|^{p - 1}u - |v|^{p - 1}v)(\tau)\, \mathrm{d}\tau \Bigr\|_Y 
 \leqslant \frac{1}{2}\|u - v\|_Y. 
} 
Therefore $\Phi : Y \to Y$ is a contraction map. 
Thus by the fixed point argument, we have the desired result.

\end{proof}

\section*{Acknowledgement}
 The author is grateful to Dr. Masayuki Hayashi JSPS Research Fellow for personal discussion of this work.  
 The author is supported by JSPS KAKENHI Grant Number 820200500051.

\end{document}